\documentclass{sl}     

\usepackage{slsec}     
\usepackage{stud_log} 
\usepackage{slfoot}
\usepackage{slthm}     
                         
\usepackage{amsmath}
\usepackage{amssymb}
\usepackage{mathrsfs}
\usepackage{enumerate}

\def\kn{\kern.1em}

\renewcommand{\theequation}{\Roman{equation}}

\usepackage{amsopn}
\DeclareMathOperator{\ob}{ob}

\newcommand{\fivee}{
	If $Y\subseteq X$ and $Z\in\ob(X)$ and $Y\cap Z\ne\emptyset$, then $Z\in\ob(Y)$.
}
\newcommand{\fived}{
	If $Y\subseteq X$ and $Y\in\ob(X)$ and $X\subseteq Z$, then $(Z\setminus X)\cup Y\in\ob(Z)$.
}
\newcommand{\fivec}{
	If $Y\in\ob(X)$ and $Z\in\ob(X)$ then $Y\cap Z\in\ob(X)$.
}
\newcommand{\fiveb}{
	If $Y\cap X=Z\cap X$ then $Y\in\ob(X)$ iff $Z\in\ob(X)$.
}
\newcommand{\fivea}{
	$\emptyset\not\in\ob(X)$.
}
\newcommand{\tru}[1]{\left\|#1\right\|} 

\newtheorem{theo}{Theorem}[section]

\theoremstyle{definition}
\newtheorem{defi}[theo]{Definition}

\HeadingsInfo{Bj\o rn Kjos-Hanssen}{A conflict between some semantic conditions of Carmo and Jones}

\begin{document}

\setcounter{page}{1}     

 

\AuthorTitle{Bj\o rn Kjos-Hanssen}{A conflict between some semantic conditions
\newline of Carmo and Jones for contrary-to-duty obligations}

   
\PresentedReceived{Richmond H.~Thomason}{July 11, 2016}


\begin{abstract}
	We show that Carmo and Jones' condition 5(e) conflicts with the other conditions on their models for contrary-to-duty obligations.
	We then propose a resolution to the conflict.
\end{abstract}


\bigskip
We shall point out a conflict between some semantic conditions proposed by \textsc{Carmo and Jones}~\cite{CJ96, CJ02, CJ13} in the area
of contrary-to-duty obligations in deontic logic. Then, we shall propose a resolution of the conflict.
For a comprehensive treatment of contrary-to-duty obligations, introduced by \textsc{Chisholm}~\cite{Chisholm1963-CHICIA-2},
one may consult~\cite{SaintCroix2014}.

\section{The conflict} 
	Let $W$ be the set of possible worlds of a given model, and let $\mathscr P$ denote the power set operation.
	\textsc{Carmo and Jones}\footnote{
		Conditions 5(a)--(d) and 5(e) are first introduced on pages 331 and 341, respectively, of~\cite{CJ96}, where ob is called pi.
		5(a)--(d) are also given in~\cite{CJ02}, page 291, with 5(e) on page 319.
		The conditions 5(a)(b)(d)(e) and a condition (c$^*$) are given in~\cite{CJ13}, page 590.
	} give the following conditions (whose numbering we retain)
	on a function $ob : \mathscr P(W) \to \mathscr P(\mathscr P(W))$ picking out that which is obligatory in a given context.
	\begin{enumerate}
		\item[5(a)] \fivea
		\item[5(b)] \fiveb
		\item[5(c)] \fivec
		\item[5(d)] \fived
		\item[5(e)] \fivee
	\end{enumerate}
	We shall now show how these conditions, in particular 5(d) and 5(e), are inherently in conflict with each other.

\begin{defi}
	Suppose $X$ and $Y$ are subsets of $W$.
	We say that $X$ and $Y$ are \emph{mutually generic}, or \emph{in general position}, if all the four sets
	\[
		X\cap Y,\qquad X\setminus Y,\qquad Y\setminus X,\qquad W\setminus (X\cup Y)
	\]
	are nonempty.
	Let $A$ and $B$ be propositions.
	Let $\tru{A}$ be the set of worlds in which $A$ is true.
	We say that $A$ and $B$ are \emph{mutually generic} if $\tru{A}$ and $\tru{B}$ are mutually generic sets.
\end{defi}
\begin{theo}\label{main}
	Suppose $A$ and $B$ are mutually generic propositions, and $\ob$ satisfies conditions 5(b), 5(d), and 5(e).
	Then
	\[
		\tru{A}\in\ob(\tru{\top}) \implies \tru{B}\in\ob(\tru{\neg A}).
	\]
\end{theo}
\begin{proof}
	Using $\tru{A\vee \neg(A\vee \neg B)} = \tru{A\vee B}$, we deduce:
	\begin{eqnarray*}
		\tru{A}\in&\ob(\tru{\top})                        & \text{(hypothesis)}\\
		\tru{A}\in&\ob(\tru{A\vee \neg B})                & \text{(by 5(e))}     \\
		\tru{A\vee \neg(A\vee \neg B)}\in&\ob(\tru{\top}) & \text{(by 5(d))}     \\
		\tru{A\vee B}\in&\ob(\tru{\top}) &      \\
		\tru{A\vee B}\in&\ob(\tru{\neg A})                & \text{(by 5(e))}     \\
		\tru{B}\in&\ob(\tru{\neg A})                      & \text{(by 5(b))}\qquad\qedhere
	\end{eqnarray*}
\end{proof}
The conclusion in Theorem~\ref{main} is patently absurd.
Just because $A$ was obligatory initially, why would we upon violation of $A$ have a contrary-to-duty obligation that $B$,
without any further information about $B$?

\bigskip
The argument just given is a refinement of a twenty-year-old one~\cite{K96}
The latter is summarized in~\cite[footnote 28]{CJ02}.
\textsc{Carmo and Jones} responded to~\cite{K96} in~\cite{CJ02}.
They argued that condition 5(c) should be weakened~\cite[p.~323]{CJ02}. Indeed, doing so takes care of the problem pointed out in~\cite{K96},
but the argument in Theorem~\ref{main} above makes no use of 5(c).

\textsc{Carmo and Jones} did see some potential for problems with the combination of just 5(b), (d), and (e)~\cite[pp.~319--320 and Figure 1]{CJ02}.
However, in the end they kept all three conditions for the system studied in~\cite{CJ13}.

If our aim were mostly destructive, then we could perhaps end this paper here. Instead, we shall now outline a possible constructive response.

\section{A resolution of the conflict}
	On our reading of \textsc{Carmo and Jones}, their conditions 5(d) and 5(e) belong to two distinct approaches.

	Both approaches start with a simpler function $F:\mathscr P(W)\rightarrow \mathscr P(W)$ which picks out the ideal worlds in a given context.
	According to the first approach, we then let
	\begin{equation}\label{sup}
		\ob(X) = \{Y: Y\supseteq F(X)\}.
	\end{equation}
	Thus, any sufficiently unrestrictive proposition will be obligatory.

	In the second approach, we let
	\begin{equation}\label{cap}
		\ob(X) = \{Y: Y\cap X = F(X)\}.
	\end{equation}
	Thus, there is essentially only one obligatory proposition in a given context.
	In both cases, we assume
	\begin{equation}\label{sub}
		F(X)\subseteq X
	\end{equation}
	and
	\begin{equation}\label{referee}
		X\ne\emptyset\implies F(X)\ne\emptyset.
	\end{equation}
	For (\ref{sup}), we consider the following two additional conditions.
	The first,
	\begin{equation}\label{can only relax}\tag{I-d}
		F(X\cap Y)\supseteq F(X)\cap Y,
	\end{equation}
	ensures 5(d). It expresses the idea that
	\begin{quote}
		our standards of perfection can only be relaxed, not strengthened, when moving to a more restricted context.
	\end{quote}
	The second,
	\begin{equation}\label{try not to relax}\tag{I-e}
		F(X\cap Y)=F(X)\cap Y\quad\text{ whenever }F(X)\cap Y\ne\emptyset,
	\end{equation}
	is a weakening of 5(e). It expresses the idea that
	\begin{quote}
		standards of perfection should only be relaxed when absolutely necessary.
	\end{quote}
	The Prisoners' Dilemma will serve to explain these conditions.
	There, four worlds are possible,
	depending on whether we and our fellow prisoner defect or not.
	Assuming we are completely selfish, these worlds are each given a numerical score, namely the number of units of
	time we must spend in prison. Among the possible contexts are
	``we and our fellow prisoner do the same thing'' and ``our fellow prisoner defects''. 
	In each case, $F(X)$ will consist exactly of those elements of $X$ in which we serve the smallest prison term.

	The first three conditions, 5(a)--(c), all hold both under (\ref{sup}) and (\ref{cap}).
	\begin{enumerate}
		\item[5(a)] \fivea
	\end{enumerate}
	For each of (\ref{sup}), (\ref{cap}) this is equivalent to $F(X)\ne\emptyset$, which follows from (\ref{referee}) when $X\ne\emptyset$.
	When $X\ne\emptyset$, we are forced by (\ref{sub}) to set $F(\emptyset)=\emptyset$.
	We do not expect that \textsc{Carmo and Jones} or anyone else would object strongly to $F(\emptyset)=\emptyset$.
	It merely concerns the ``corner case'' of the impossible context, in which we need not insist that anything is obligatory.
	\begin{enumerate}
		\item[5(b)] \fiveb
	\end{enumerate}
	This follows from (\ref{sub}) for both cases (\ref{sup}) and (\ref{cap}).
	\begin{enumerate}
		\item[5(c)] \fivec
	\end{enumerate}
	True for both (\ref{sup}) and (\ref{cap}).

	Things get more interesting with 5(d)--(e).
	\begin{enumerate}
		\item[5(d)] \fived
	\end{enumerate}
	Under (\ref{cap}), 5(d) is plainly
	unreasonable, as $Z\setminus X$ may contain many non-ideal worlds.
	\begin{theo}
		Under (\ref{sup}), 5(d) becomes true if we add condition (\ref{can only relax}).
	\end{theo}
	\begin{proof}
		Suppose $Y\subseteq X$, $Y\in\ob(X)$, and $X\subseteq Z$.
		By (\ref{sup}), $Y\supseteq F(X)$.
		By (\ref{can only relax}), $F(Z)\cap X\subseteq F(Z\cap X)=F(X)$.
		Hence
		\begin{eqnarray*}
			F(Z) = F(Z)\cap Z &=& \left(F(Z)\cap X\right) \cup \left(F(Z)\cap (Z\setminus X)\right)\\
			                  &\subseteq&               Y \cup (Z\setminus X),
		\end{eqnarray*}
		as desired.
	\end{proof}
	We now consider
	\begin{enumerate}
		\item[5(e)] \fivee
	\end{enumerate}
	\begin{theo}
		Under (\ref{cap}), 5(e) becomes true if we add condition (\ref{try not to relax}).
	\end{theo}
	\begin{proof}
		Suppose $Y\subseteq X$, $Z\in\ob(X)$, and $Y\cap Z\ne\emptyset$.
		By (\ref{cap}), $Z\cap X = F(X)$.
		By (\ref{try not to relax}), $F(Y) = F(Y\cap X)=F(X)\cap Y$ since $F(X)\cap Y\ne\emptyset$.
		Hence $Z\cap Y=F(Y)$, meaning by (\ref{cap}) that $Z\in\ob(Y)$, as desired.
	\end{proof}
	Under (\ref{sup}), 5(e) must be weakened. We find that much of the same spirit is retained by using (\ref{try not to relax}).

	\bigskip

	We could now go back to the rest of \textsc{Carmo and Jones}' system and explore two notions,
	corresponding to (\ref{sup}) and (\ref{cap}) respectively, of ``weakly'' and ``strongly'' obligatory propositions.
	We leave that for future research.

	If we define $A\Rightarrow B$ (or in another notation, $O(B\mid A)$) by $\tru{B}\in\ob(\tru{A})$ then we recover a \emph{standard conditional model} in the sense of Chellas~\cite[Chapter 10]{Chellas}; see also~\cite{LewisJournal},~\cite{Lewis73}.
	In fact, it is a special case of standard conditional model in which the truth of $A\Rightarrow B$ does not depend on the actual world.

	Finally, we mention a relationship with
	preference-based approaches (see \cite{PS97} and \cite[Section 7.3]{CJ02}).
	Our function $F$ induces a preference order by letting
    \[
        a\le b\iff (\forall X)(a\in F(X)\rightarrow b\in F(X)).
    \]
	Thus, $a$ is \emph{no more desirable than} $b$ ($a\le b$) if in any context $X$, if $a$ is ideal given $X$ then $b$ is ideal given $X$.
	This ordering can also be localized to a context:
    \[
        a\le_Y b\iff (\forall X\subseteq Y)(a\in F(Y)\rightarrow b\in F(Y)).
    \]
	While \textsc{Prakken and Sergot}~\cite{PS97} argued that a preference ordering on worlds is essential in order to cope with CTD-phenomena,
	\textsc{Carmo and Jones} concluded that it is not~\cite[page 343]{CJ02}.
	Instead they conjecture that
	``the concept(s) of \emph{settledness} provide the fundamental key to unravelling the tangled knot of CTD-problems'' [\emph{loc. cit.}, emphasis ours].

	We find some middle ground between these two conclusions.
	A preference ordering does appear, but it does so as a byproduct of a concept of settledness: the context $X$ is a settled fact when considering the ideal worlds in $F(X)$.

\paragraph{Acknowledgements.}
	This work was partially supported by a grant from the Simons Foundation (\#315188 to Bj\o rn Kjos-Hanssen). 
	This material is based upon work supported by the National Science Foundation under Grant No.\ 1545707.
	We would like to thank Professor Andrew J.I.~Jones for serving as adviser on the project leading to~\cite{K96}.
\newpage
\bibliographystyle{sl}
\bibliography{filhspe1}

\begin{thebibliography}{10}

\bibitem{CJ96}
\textsc{Carmo, Jos{\'e} M. C. L.~M.}, and \textsc{Andrew J.~I. Jones}, `A new
  approach to contrary-to-duty obligations', in D.~Nute, (ed.),
  \emph{Defeasible Deontic Logic}, vol. 263 of \emph{Synthese Library}, Kluwer
  Academic Publishers, 1997, pp. 317--344.

\bibitem{CJ02}
\textsc{Carmo, Jos{\'e} M. C. L.~M.}, and \textsc{Andrew J.~I. Jones}, `Deontic
  logic and contrary-to-duties', in D.M. Gabbay, and F.~Guenthner, (eds.),
  \emph{Handbook of Philosophical Logic}, vol.~8, 2nd edn., Kluwer Academic
  Publishers, 2002, pp. 265--343.

\bibitem{CJ13}
\textsc{Carmo, Jos{\'e} M. C. L.~M.}, and \textsc{Andrew J.~I. Jones},
  `Completeness and decidability results for a logic of contrary-to-duty
  conditionals', \emph{J. Logic Comput.}, 23 (2013), 3, 585--626.

\bibitem{Chellas}
\textsc{Chellas, Brian~F.}, \emph{Modal logic}, Cambridge University Press,
  Cambridge-New York, 1980. An introduction.

\bibitem{Chisholm1963-CHICIA-2}
\textsc{Chisholm, Roderick}, `Contrary-to-duty imperatives and deontic logic',
  \emph{Analysis}, 24 (1963), 2, 33--36.

\bibitem{K96}
\textsc{Kjos-Hanssen, Bj{\o}rn}, `Models of the {C}hisholm set',  (1996).
  Philosophy term paper, University of Oslo. arxiv 1607.02189.

\bibitem{LewisJournal}
\textsc{Lewis, David}, `Counterfactuals and comparative possibility', \emph{J.
  Philos. Logic}, 2 (1973), 4, 418--446.

\bibitem{Lewis73}
\textsc{Lewis, David}, \emph{Counterfactuals}, Blackwell Publishers, Oxford,
  2001. Revised reprinting of the 1973 original.

\bibitem{PS97}
\textsc{Prakken, Henry}, and \textsc{Marek Sergot}, `Dyadic deontic logic and
  contrary-to-duty obligations', in D.~Nute, (ed.), \emph{Defeasible Deontic
  Logic}, vol. 263 of \emph{Synthese Library}, Kluwer Academic Publishers,
  1997, pp. 223--262.

\bibitem{SaintCroix2014}
\textsc{Saint~Croix, Catharine}, and \textsc{Richmond~H. Thomason},
  \emph{Chisholm's Paradox and Conditional Oughts}, Springer International
  Publishing, Cham, 2014, pp. 192--207.

\end{thebibliography}

\AuthorAdressEmail{Bj\o rn Kjos-Hanssen}{Department of Mathematics\\
University of Hawai\textquoteleft i at M\=anoa\\
2565 McCarthy Mall\\
Honolulu, U.S.A.}{bjoern.kjos-hanssen@hawaii.edu}

\end{document}